\documentclass{article}
\usepackage{amsfonts}
\usepackage{latexsym}
\usepackage{amsmath}
\usepackage{amssymb}
\usepackage{amscd}
\usepackage{epsfig}
\usepackage{colortbl}
\usepackage{xcolor}
\addtocounter{MaxMatrixCols}{4}
\newtheorem{definition}{Definition}[section]

\newtheorem{lemma}[definition]{Lemma}
\newtheorem{theorem}[definition]{Theorem}

\newtheorem{conjecture}[definition]{Conjecture}
\newcommand{\qed}{\hskip 10pt $\Box$}
\newenvironment{proof}{\par {\sc {\bf Proof.}\hskip 5pt}}{\hfill \qed \par}
\newcommand{\undertilde}[1]{\ensuremath{\mathord{\vtop{\ialign{##\crcr
   $\hfil\displaystyle{#1}\hfil$\crcr\noalign{\kern1.5pt\nointerlineskip}
   $\hfil\tilde{}\hfil$\crcr\noalign{\kern1.5pt}}}}}}

\begin{document}

\title{On the Minimum Number of Hamiltonian Cycles in Regular Graphs}
\author{M. Haythorpe\thanks{Flinders University, Australia. Tel.: +618-820-12375. Email: michael.haythorpe@flinders.edu.au}}

 \maketitle

\begin{abstract}A graph construction that produces a $k$-regular graph on $n$ vertices for any choice of $k \geq 3$ and $n = m(k+1)$ for integer $m \geq 2$ is described. The number of Hamiltonian cycles in such graphs can be explicitly determined as a function of $n$ and $k$, and empirical evidence is provided that suggests that this function gives a tight upper bound on the minimum number of Hamiltonian cycles in $k$-regular graphs on $n$ vertices for $k \geq 5$ and $n \geq k+3$. An additional graph construction for 4--regular graphs is described for which the number of Hamiltonian cycles is superior to the above function in the case when $k = 4$ and $n \geq 11$.\end{abstract}

\section{Introduction}

The Hamiltonian cycle problem (HCP) is a, now classical, graph theory problem that can be stated as follows: given a graph $G$ on $n$ vertices, determine whether any simple cycles of length $n$ exist in $G$. Such simple cycles are called {\em Hamiltonian cycles}, and graphs which contain at least one Hamiltonian cycle are called {\em Hamiltonian}. HCP is known to be NP-complete even when restricted to regular graphs \cite{tarjan}. The related problem of determining the number of Hamiltonian cycles in a graph is a \#P-complete problem.

There have been a number of results and conjectures regarding upper and lower bounds on the numbers of Hamiltonian cycles in regular graphs. Eppstein \cite{eppstein} conjectured that 3-regular graphs on $n$ vertices have at most $2^{n/3}$ Hamiltonian cycles, and provided a family of graphs meeting this bound for each positive $n = 0\mod 6$. While the validity of this conjecture has not yet been determined, empirical evidence indicates it is likely to be true. Gebauer \cite{gebauer} has provided the best proven upper bound of $1.276^n$. Gebauer's construction is also generalised for regular graphs of higher degree, providing a graph with $\left( \lceil \frac{k-1}{2} \rceil (k-1)! (k-2)! \lfloor \frac{k+1}{2} \rfloor \right)^{\frac{n}{2k}}$ for any $k \geq 3$ and $n$ divisible by $2k$.

In terms of regular graphs with few Hamiltonian cycles, it is known that Hamiltonian 3-regular graphs must contain at least three HCs \cite{tutte}, and there are infinite families of 3--regular graphs with this number of HCs, such as the generalized Petersen graphs \cite{watkins} GP(n,2) for $n = 3 \mod 6$. However, such bounds are not known for $k$-regular graphs for $k \geq 4$, and in fact, it is still not known whether {\em uniquely Hamiltonian} $k$-regular graphs exist (that is, $k$-regular graphs with only a single Hamiltonian cycle), although it can be seen from \cite{tutte} that if such a graph were to exist, it must be for even $k$, and further from \cite{haxell} that $k \leq 22$. Empirical evidence indicates that such a graph is extremely unlikely to exist.

In this manuscript, a construction is provided that produces a graph for any choice of $k \geq 3$, and $n = m(k+1)$ for integer $m \geq 2$. It will be shown that the number of Hamiltonian cycles in such a graph is $h(n,k) := (k-1)^2\left[(k-2)!\right]^{m}$, which we will observe appears to be a small value for $k$-regular graphs. Note that $h(n,3) = 4$ irrespective of the choice of $n$, but this result is dominated by the presence of 3-regular graphs with three HCs. For $k = 4$ an alternative construction will be provided for which the number of Hamiltonian cycles is less than $h(n,4)$ for all $n \geq 11$. For any choice of $k \geq 5$, however, it will be conjectured that there are at most two $k$-regular graphs with fewer HCs than $\lceil h(n,k) \rceil$. In particular, for odd $k \geq 5$ there will be precisely one of order $k+1$, and for even $k > 5$ there will be two, of order $k+1$ and $k+2$ respectively.

\section{Construction}

For any choice of $k \geq 3$ and $n = m(k+1)$ for integer $m \geq 2$, a graph $G_{n,k}$ can be constructed by the following scheme.

\begin{enumerate}\item Take $m$ copies of $K_{k+1}$, and label them $K^1_{k+1}, K^2_{k+1}, \hdots, K^m_{k+1}$.
\item Remove one edge in each of $K^1_{k+1}$ and $K^m_{k+1}$. In each case the removed edge was incident on two vertices. Label the two vertices in $K^1_{k+1}$ as $a_1$ and $v_1$, and the two vertices in $K^m_{k+1}$ as $b_m$ and $v_m$.
\item Remove two edges incident on a common vertex in each of $K^2_{k+1}, \hdots, K^{m-1}_{k+1}$. In each case, label the common vertex $v_j$ for $j = 2, \hdots, m-1$. Also, each removed edge was incident on another vertex, label these $a_j$ and $b_j$ for $j = 2, \hdots, m-1$.
\item Add edges $(v_i,v_{i+1})$ and $(a_i,b_{i+1})$ for all $i = 1, \hdots, m-1$.
\end{enumerate}

An example of the above construction for $k = 5$ and $m = 4$ can be seen in Figure \ref{fig-example}.

\begin{figure}[h!]\begin{center}\includegraphics[scale=0.3]{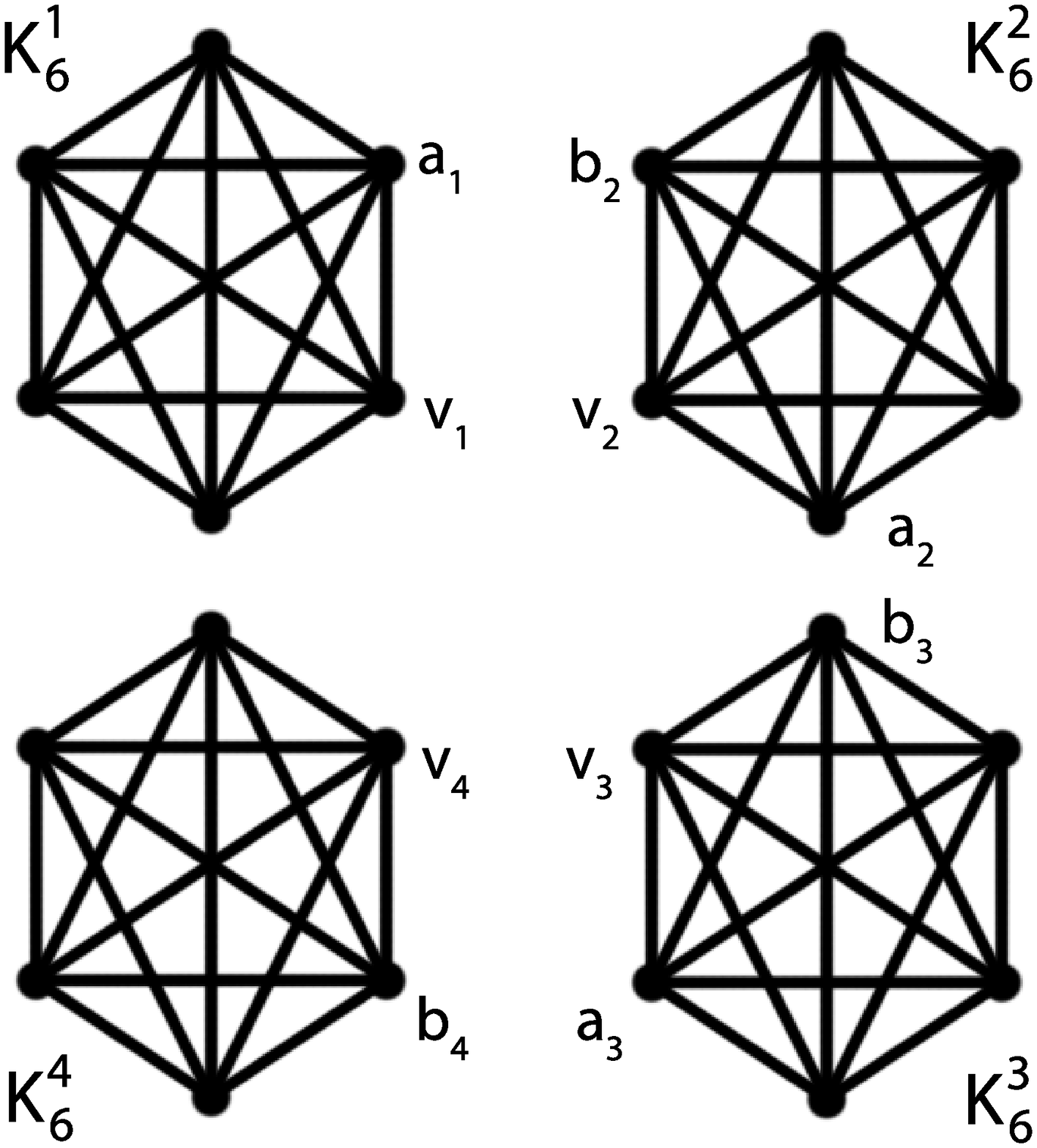} \hspace*{2cm} \includegraphics[scale=0.3]{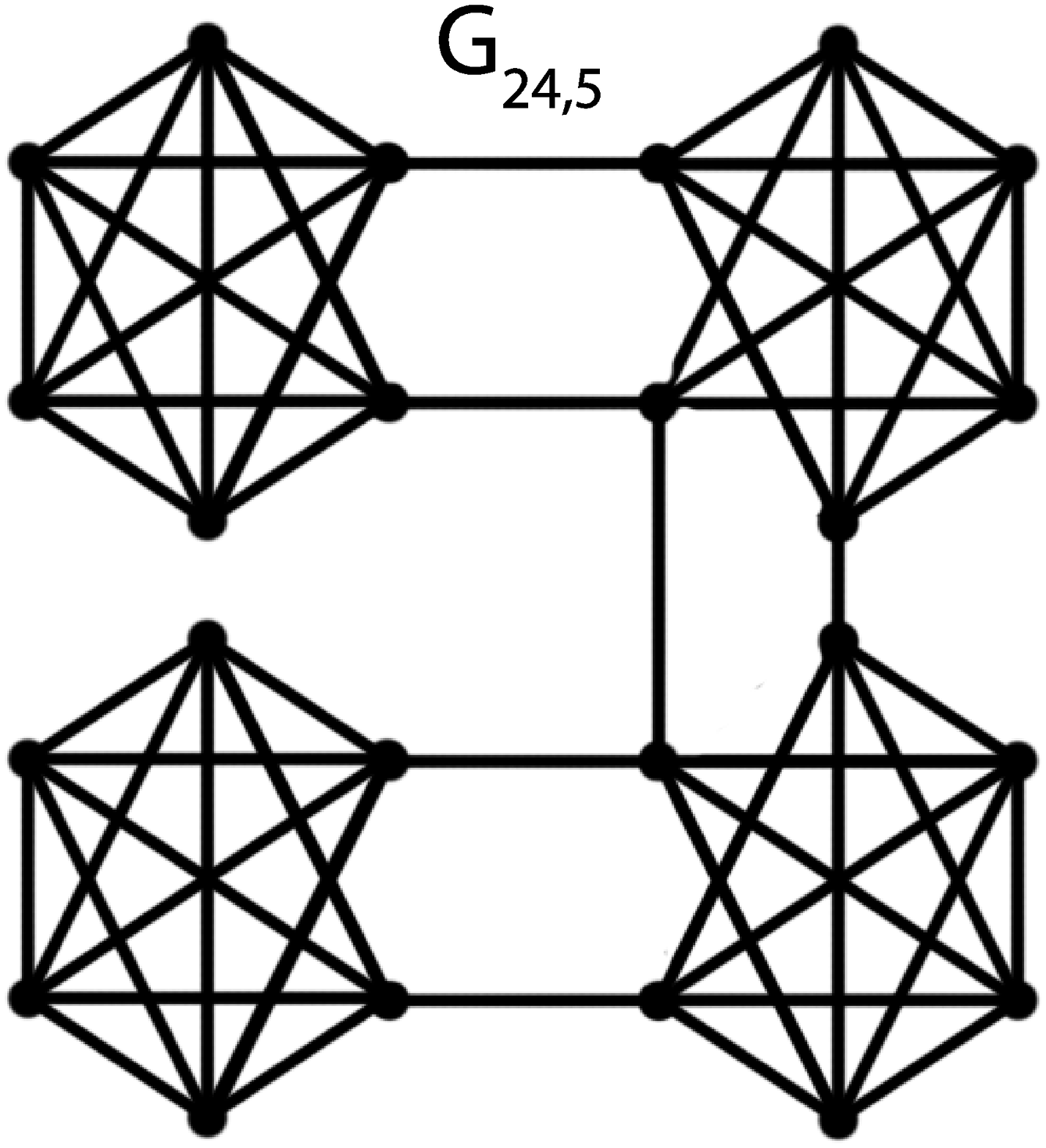}\caption{An example of the above construction for $k = 5$ and $m = 4$. The four component graphs are displayed on the left, and the constructed graph $G_{24,5}$ is displayed on the right.}\label{fig-example}\end{center}\end{figure}

It is easy to check that $G_{n,k}$ constructed in the above way is connected and $k$-regular. Next, we determine the number of HCs in $G_{n,k}$, which we denote by $h(n,k)$.

\begin{theorem}The number of Hamiltonian cycles in $G_{n,k}$ is

$$h(n,k) = (k-1)^2\left[(k-2)!\right]^{\frac{n}{k+1}}.$$\label{thm-numhc}\end{theorem}

\begin{proof}Recall that $m = \frac{n}{k+1} \geq 2$, and note that edges $(v_i,v_{i+1})$ and $(a_i,b_{i+1})$ constitute edge cutsets of size 2 for each $i = 1, \hdots, m-1$. Hence, all of those edges must be included in every HC of $G_{n,k}$. In particular, this means that in the HC, each $v_i$ is visited between $v_{i-1}$ and $v_{i+1}$ for $i = 2, \hdots, m-1$.

Then, consider an HC beginning at $v_m$. The HC then travels through all $v_i$ in reverse order until it reaches $v_1$. At this point, the rest of $K^1_{k+1}$ must be traversed before visiting vertex $a_1$. There are $(k-1)!$ ways of doing so.

Then, each $K^j_{k+1}$ is visited in order for $j = 2, \hdots, m-1$. Each time, it is entered via vertex $b_j$. Then, the only remaining edge by which it may be departed emanates from $a_j$, so this vertex must be visited last, and vertex $v_j$ has previously been visited. This leaves $(k-2)!$ possible ways of traversing each $K^j_{k+1}$. Finally, $K^m_{k+1}$ is revisited via vertex $b_m$, and there remains $(k-1)!$ ways of visiting the remaining vertices before returning to $v_m$ and completing the HC.

It is then clear that

\begin{eqnarray*}h(n,k) & = & \left[(k-1)!\right]^2\left[(k-2)!\right]^{m-2}\\
& = & \frac{\left[(k-1)!\right]^2\left[(k-2)!\right]^m}{\left[(k-2)!\right]^2}\\
& = & (k-1)^2\left[(k-2)!\right]^{\frac{n}{k+1}}.\end{eqnarray*}\end{proof}

\section{Empirical Results}

%

As mentioned previously, $h(n,3) = 4$ for all values of $n$, so there are infinitely many 3-regular graphs with fewer Hamiltonian cycles than $G_{n,3}$. However, for some choices of $k \geq 4$, this may not be the case. To investigate this, the minimum number of Hamiltonian cycles in $k$-regular graphs of order $n$ was computed for various small values of $k$ and $n$. Those minimum numbers are displayed in Table \ref{tab-minhcs}, and they indicate that, at least for small values of $n$ where such calculations are tractable, there are very few graphs with fewer than $\lceil h(n,k) \rceil$ Hamiltonian cycles. The numbers in Table \ref{tab-minhcs} were computed by first using GENREG \cite{genreg} to construct all $k$-regular graphs on $n$ vertices for various values of $n$ and $k$, and then using the Hamiltonian cycle enumeration algorithm by Chalaturnyk \cite{chalaturnyk} to count the number of Hamiltonian cycles in each. In order to speed up the computations, the sets of graphs were partitioned into manageable subsets and distributed over 800 cores. In all cases checked exhaustively, the minimal example had vertex connectivity 2 if any such graphs existed for that choice of $k$ and $n$ (that is, for $n \geq 2k$). Indeed, it was typical for not only the minimal graph to have vertex connectivity 2, but the vast majority of \lq\lq near-minimal" graphs as well. To that end, in the cases where the number of graphs to be checked exceeded 100 million (respectively, $n \geq 17$ for $k = 4$, $n \geq 16$ for $k = 5$, $n \geq 15$ for $k = 6, 7$), only the graphs with vertex connectivity 2 were checked. Those figures are in boxes shaded grey to distinguish them from the values calculated exhaustively. 

\begin{table}[h]
\begin{center}
\begin{tabular}{|c|c|c|c|c|}
\hline  \;\;\; $k$ & 4 & 5 & 6 & 7 \\
$n$ & & & & \\
\hline 5 & 12 & - & - & -\\
\hline 6 & 16 & 60 & - & -\\
\hline 7 & 23 & - & 360 & -\\
\hline 8 & 29 & 177 & 744 & 2520\\
\hline 9 & 36 & - & 1553 & -\\
\hline 10 & 36 & 480 & 3214 & 14963\\
\hline 11 & 48 & - & 6564 & - \\
\hline 12 & 60 & 576 & 12000 & 87808\\
\hline 13 & 72 & - & 22680 & - \\
\hline 14 & 72 & 1296 & 14400 & 430920\\
\hline 15 & 72 & - & \cellcolor{lightgray} 29760 & - \\
\hline 16 & 72 & \cellcolor{lightgray} 3888 & \cellcolor{lightgray} 57600 & \cellcolor{lightgray} 518400 \\
\hline 17 & \cellcolor{lightgray} 96 & - & \cellcolor{lightgray} 118080 & - \\
\hline 18 & \cellcolor{lightgray} 108 & \cellcolor{lightgray} 3456 & \cellcolor{lightgray} 239040 & \cellcolor{lightgray} 2937600 \\
\hline

\end{tabular} \;\;\; \begin{tabular}{|c|c|c|c|c|}
\hline \;\;\; $k$ & 4 & 5 & 6 & 7\\
$n$ & & & & \\
\hline 5 & 18 & - & - & -\\
\hline 6 & 21 & 96 & - & -\\
\hline 7 & 24 & - & 600 & -\\
\hline 8 & 28 & 175 & 945 & 4320\\
\hline 9 & 32 & - & 1488 & -\\
\hline 10 & 36 & 317 & 2343 & 14299\\
\hline 11 & 42 & - & 3689 & - \\
\hline 12 & 48 & 576 & 5808 & 47324\\
\hline 13 & 55 & - & 9146 & - \\
\hline 14 & 63 & 1047 & 14400 & 156629\\
\hline 15 & 72 & - & 22675 & - \\
\hline 16 & 83 & 1902 & 35704 & 518400\\
\hline 17 & 96 & - & 56219 & - \\
\hline 18 & 110 & 3456 & 88523 & 1715775\\
\hline

\end{tabular}
\end{center}
\caption{The first table displays the minimum number of Hamiltonian cycles for all $k$-regular graphs on $n$ vertices, and the second table displays $\lceil h(n,k) \rceil$. Entries of '-' indicate no graphs exist for that choice of $n$ and $k$. Shaded values indicate that only graphs with vertex connectivity 2 were checked for that choice of $k$ and $n$.}
\label{tab-minhcs}\end{table}

In analysing Table \ref{tab-minhcs}, first consider the case where $k > 4$. The only graphs with fewer Hamiltonian cycles than $\lceil h(n,k) \rceil$ for $k > 4$ discovered during testing were $K_{k+1}$ for $k \geq 4$, and the {\em cocktail party graph} \cite{biggs} of order $\frac{k}{2} + 1$ for even $k \geq 4$. The complete graph $K_{k+1}$ contains $\frac{k!}{2}$ Hamiltonian cycles which is strictly less than $h(k+1,k) = (k-1)^2(k-2)!$ for all $k \geq 2$.

The cocktail party graph of order $l$ is equivalent to a complete graph on order $2l$ minus a perfect matching. The cocktail party graph of order 3 is displayed in Figure \ref{fig-cocktail3}. 

\begin{figure}[h!]\begin{center}\includegraphics[scale=0.3]{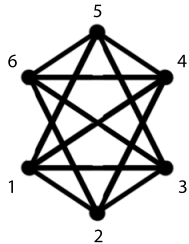}\caption{The cocktail party graph of order 3.}\label{fig-cocktail3}\end{center}\end{figure}

The number of Hamiltonian cycles in the cocktail party graph is equivalent to the solution of the so-called {\em relaxed m\'{e}nage problem}, which requests the number of ways to seat a given number of couples around a circular table so that nobody sits next to their partner, up to symmetry. Hence, we can use the result in \cite{relaxed} to determine the number of Hamiltonian cycles in the cocktail party graph of order $\frac{k}{2} + 1$:

$$\sum_{i=0}^{\frac{k}{2} + 1} (-1)^i \left(\begin{array}{c}\frac{k}{2} + 1\\i\end{array}\right) (k-i+1)! 2^{i-1}.$$

Using the link between the relaxed m\'{e}nage problem and the Hamiltonian cycle of the $n$-octohedron, we can take advantage of a result in \cite{singmaster} to see that the number of Hamiltonian cycles in the cocktail party graph of order $\frac{k}{2} + 1$ is strictly less than $\frac{(k+1)!}{2e}$. Then Stirling's formula for the approximation of factorials \cite{stirling} can be used to show that $h(k+2,k) > \frac{(k+1)!}{2e}$ for all $k \geq 6$. Hence, cocktail party graphs will always have fewer HCs than the equivalent sized graph arising from the construction in the previous section. Note that cocktail party graphs are always regular graphs of even order.

Note also that, for even $k$, the two graphs described above are the only two graphs that exist on order $n \leq k+2$. For odd $k$, only $K_{k+1}$ exists for $n \leq k+2$.

The analysis for the case where $k = 4$ reveals that, in addition to $K_5$ and the cocktail party graph of order 3, there were two more graphs found that have fewer Hamiltonian cycles than suggested by $\lceil h(n,4) \rceil$. The first is the circulant graph Ci$_7(1,2)$ which has 23 Hamiltonian cycles, which is less than $h(7,4) \approx 23.75$. The second is a graph on 16 vertices which is the first member of a new family of graphs described in the following section that is, to the best of the author's knowledge, the best known construction for 4--regular graphs with minimal Hamiltonian cycles.

The above analysis leads to the following conjecture.

\begin{conjecture}For $k \geq 5$ and $n \geq k+3$, all $k$-regular graphs have at least $h(n,k)$ Hamiltonian cycles.\label{conj1}\end{conjecture}

\section{Construction for 4--regular graphs}

A family of 4--regular graphs $G^*_{n,4}$ on $n = 10 + 6m$ vertices, for any $m \geq 1$, can be constructed as follows.

\begin{enumerate}\item Take two copies of $K_5$, and label them $P_1$ and $P_{m+2}$. Also take $m$ copies of the cocktail party graph of order 3, and label them $P_2, \hdots, P_{m+1}$. In each copy, label the vertices as in Figure \ref{fig-cocktail3}.
\item Remove edge $(1,2)$ in $P_1$, and edge $(1,5)$ in $P_{m+2}$.
\item Remove edges $(1,2)$ and $(1,5)$ in all copies of $P_2, \hdots, P_{m+1}$.
\item For each $i = 1, \hdots, m+1$, label vertex 1 from $P_i$ as $a$, vertex $2$ from $P_i$ as $b$, vertex $1$ from $P_{i+1}$ as $c$ and vertex $5$ from $P_{i+1}$ as $d$. Then add edges $(a,c)$ and $(b,d)$.
\end{enumerate}

The above construction is illustrated for $m = 2$ in Figure \ref{fig-quart}.

\begin{figure}[h!]\begin{center}\includegraphics[scale=0.25]{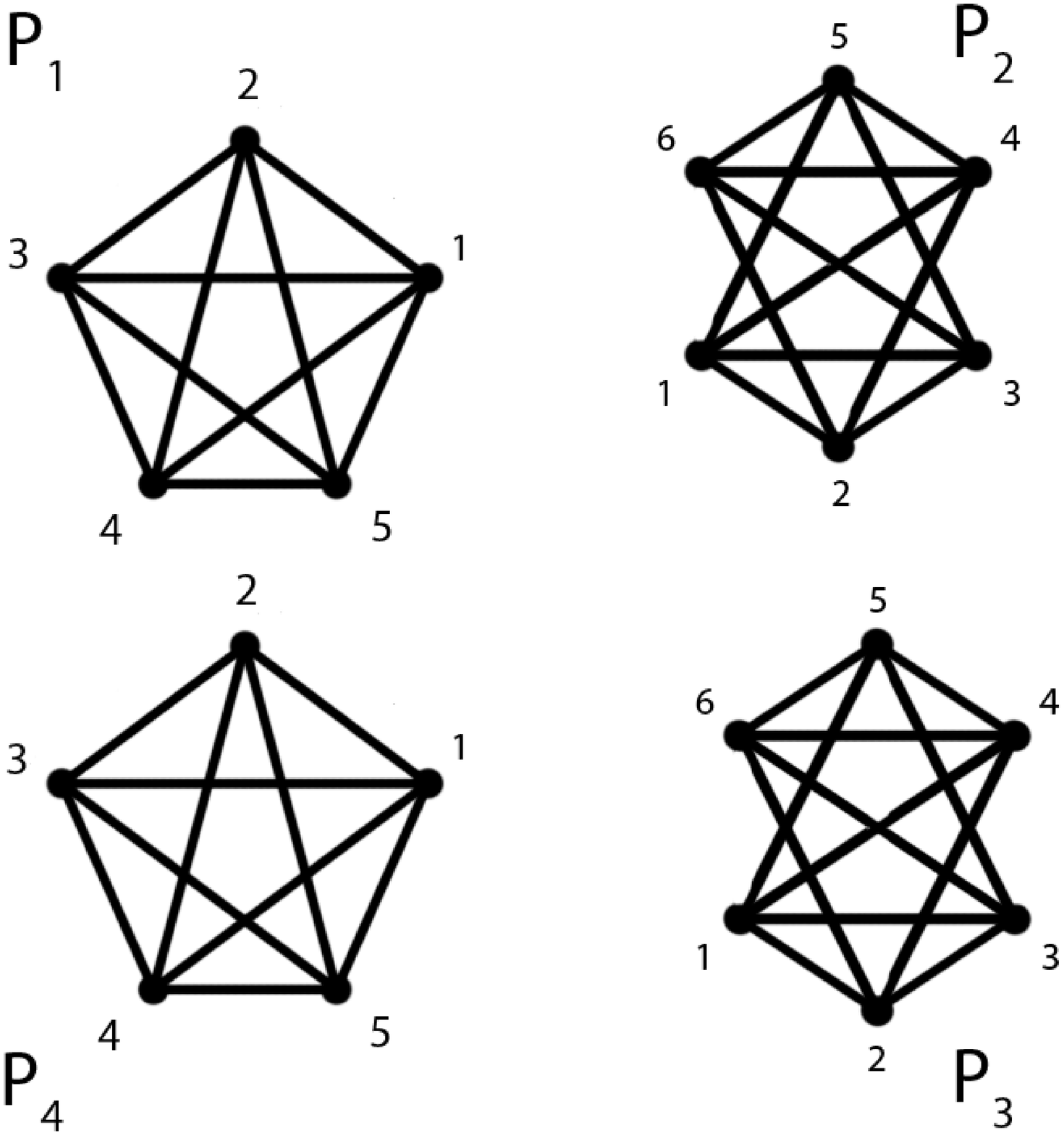} \hspace*{2cm} \includegraphics[scale=0.25]{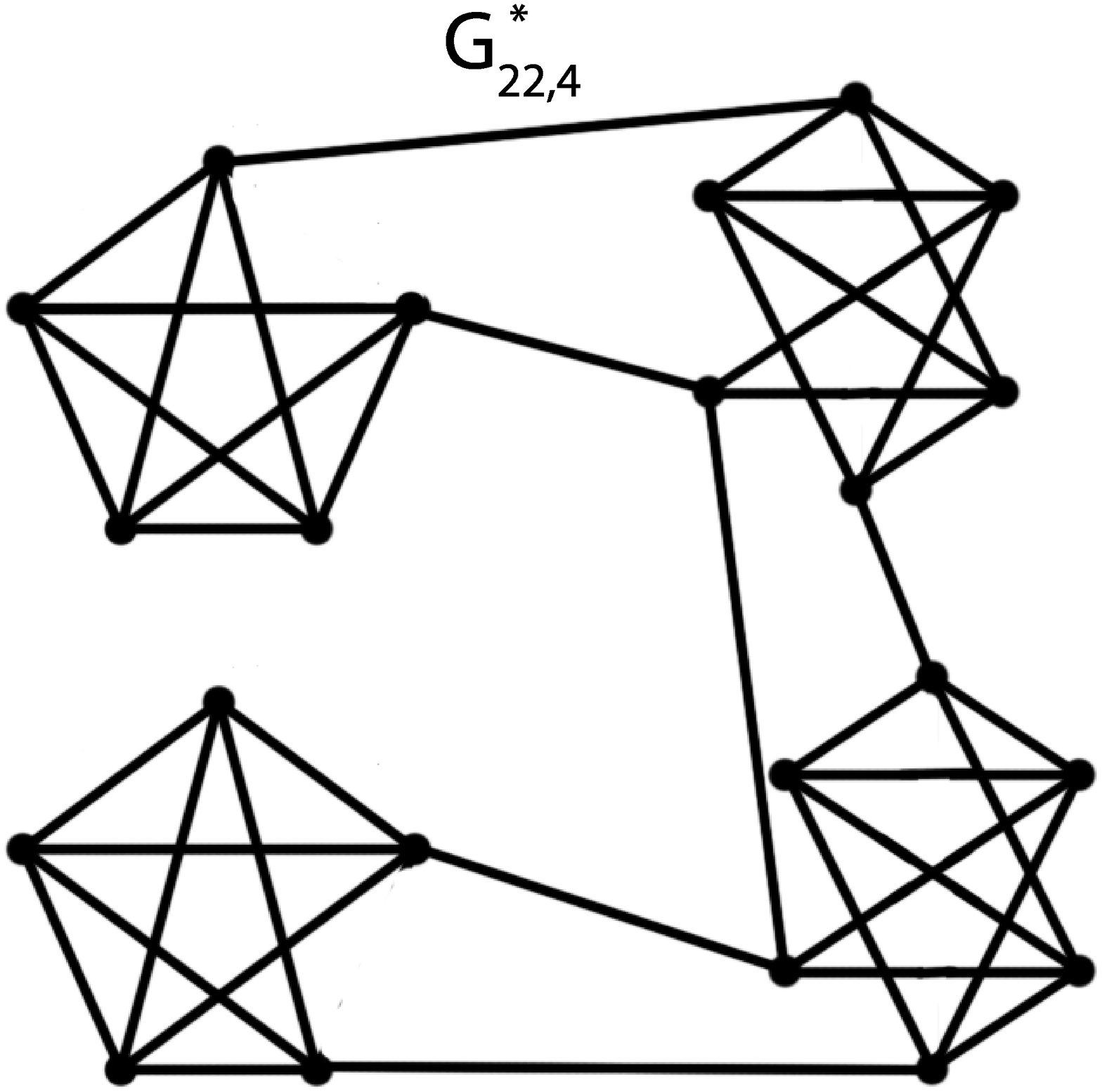}\caption{An example of the 4--regular construction for $m = 2$. The four component graphs are displayed on the left, and the constructed graph $G^*_{22,4}$ is displayed on the right.}\label{fig-quart}\end{center}\end{figure}

Using a similar argument to the proof of Theorem \ref{thm-numhc}, it can be seen that all Hamiltonian cycles in a graph constructed in this manner must use every one of the added edges. Then, it is easy to check that each of the $m$ copies of the cocktail party graph of order 3 can only be validly traversed in two ways each, and both copies of $K_5$ in six ways each. Hence, the following result emerges immediately.

\begin{lemma}The number of Hamiltonian cycles in $G^*_{n,4}$ for $n = 10 + 6m$ and $m \geq 1$ is $36(2)^m = 9(2)^{\frac{n+2}{6}}$.\end{lemma}

It is easy to see that this figure is less than $h(n,4) = 9(2)^{\frac{n}{5}}$ for all $n \geq 11$.

The above construction can be generalised for even $k > 4$, using two copies of $K_{k+1}$ and $m$ copies of the cocktail party graph of order $\frac{k}{2} + 1$, but such graphs do not appear to have fewer than $h(n,k)$ Hamiltonian cycles. Experiments for small $k$ indicate that the best such constructions are equivalent to the above except the order of cocktail party graphs used is $(k+2)/2$, and edges $(1,2)$ and $(1,k+1)$ are to be removed in place of $(1,2)$ and $(1,5)$ respectively. For $k = 6$ this produces graphs with $14400(48)^{\frac{n-14}{8}}$ Hamiltonian cycles for $n = 14 + 8m$ vertices, and for $k = 8$ this produces graphs with $25401600(1968)^{\frac{n-18}{10}}$ vertices. In both cases, these values are greater than $h(n,k)$ for all positive $n$, indicating that this construction is likely to be only useful for $k = 4$. This manuscript is now concluded with a second conjecture. It is worth noting that, should Conjectures \ref{conj1} and \ref{conj2} be proved correct, it would answer in the affirmative the long-standing conjecture \cite{sheehan} that no $r$-regular graphs are uniquely hamiltonian for $r > 2$.

\begin{conjecture}For $n \geq 8$, all 4-regular graphs of order $n$ have at least $9(2)^{\frac{n+2}{6}}$ Hamiltonian cycles.\label{conj2}\end{conjecture}

\end{document}